\documentclass[11pt]{amsart}

\usepackage{amsthm}
\usepackage{amsmath}
\usepackage{amssymb}

\usepackage[text={6in,9in},centering]{geometry}

\newcommand{\tr}{\mathop\mathrm{tr}}

\newtheorem{thm}{Theorem}[section]

\newtheorem{corollary}[thm]{Corollary}

\newtheorem{lemma}[thm]{Lemma}

\newtheorem{definition}[thm]{Definition}
\theoremstyle{remark}

\newcommand{\cH}{\mathcal H}

\title{Random fusion frames are nearly equiangular and tight}
\author[B.G. Bodmann]{Bernhard G. Bodmann}
\address{Department of Mathematics, University of Houston, Houston, TX 77204-3008}
\email{bgb@math.uh.edu}

\thanks{The research presented in this paper was supported by NSF DMS 1109545 and 
by AFOSR FA9550-11-1-0245.}

\begin{document}

\begin{abstract}
This paper demonstrates 
that random, independently chosen equi-di\-men\-sion\-al subspaces with a unitarily invariant distribution in a real Hilbert space provide nearly tight, nearly equiangular fusion frames.
The angle between a pair of subspaces is measured in terms of the Hilbert-Schmidt inner product of the
corresponding orthogonal projections. If the subspaces are selected at random, then a measure concentration
argument shows that these inner products concentrate near an average value. Overwhelming success probability 
for near tightness and equiangularity is
guaranteed if the dimension of the subspaces is sufficiently small compared to that of the Hilbert space and if
the dimension of the Hilbert space is small compared to the sum of all subspace dimensions.
\end{abstract}

\maketitle

\section{Introduction}

A collection of closed  subspaces in a Hilbert space is a fusion frame if a weighted sum of the corresponding
 orthogonal projections provides an approximate identity. 
Research on fusion frames has enjoyed a rapid growth in the frame theory literature of the last decade, see \cite{FFbook}
and references therein. There are
many applications of this field, driven by demands from 
distributed sensing \cite{BKR11}, parallel processing \cite{Bjorstad:1991kx}, communication theory \cite{B}, quantum computing \cite{Bodmann:2007kx} and even neuroscience \cite{Rozell:2006uq}. Fusion frames are important 
for these applications
because they model linear, distributed signal representation strategies.
A primary goal in many of these settings is to make the representation robust to partial data loss when
a signal to be stored or transmitted
is projected onto the subspaces given by a fusion frame. If the components in the subspaces are 
lost then this constitutes an erasure. The performance of a fusion frame 
is typically measured by its ability to compensate for such lost components,
either in a deterministic, adversarial erasure regime or in an averaged sense
with additional statistical assumptions on the signal and on the erasures.
Earlier work on fusion frame design and erasures
has shown that typical notions of optimality have a geometric
characterization: If equi-isoclinic fusion frames exist, then they are optimal for the worst-case error for recovery based on the canonical dual when up to two subspace components are removed,
and if equi-distant fusion frames exist, then they are optimal for the mean-squared error after applying a Wiener
filter to suppress the effect of erasures, see also a simpler setting in which they provide optimality \cite{BB-FOC}. 
The existence and construction of such subspace collections depends on the dimensions and
has remained a challenge despite many efforts \cite{LS,H,GH89,SPLAG,ET06,ET07,ETF09,BCPST10,CCHKP11}.
The motivation for this paper is to show that  a random choice of subspaces is with high probability nearly tight and nearly equiangular.
In a certain regime, this was already accomplished  in an earlier paper by the construction of nearly tight frames with nearly isoclinic
subspaces from matrices with the restricted isometry \cite{BCC12}, but this only resulted in nearly orthogonal subspaces. 
It somehow seems natural that a random selection favors mutual orthogonality which is in a way the simplest case of equiangularity.
A generalized form of the Welch bound shows that an equiangular, equi-dimensional tight
fusion frame minimizes the maximum value for the Hilbert-Schmidt inner product of
any pair of orthogonal projections $P$ and $P'$ corresponding to two subspaces, resulting in
the value of the inner product \cite{KPCL09,BE}
$$
   \tr[ P P' ] = \frac{s(Ks-N)}{(K-1)N} \, .
$$
As $N, K, s \to \infty$, $N \tr[P P']/s^2 \to 1$.   We establish these asymptotics for
a random choice of subspaces.
We cover this case of non-orthogonal equiangular subspaces by further developing the partial orthonormalization strategy
 applied to frames that can be partitioned into nearly tight Riesz sequences  from \cite{BCC12}.
Such frames are implicitly constructed in the compressed sensing literature as the column vectors of sensing matrices
that have the restricted isometry property. Many of the standard construction methods rely on randomization,
meaning they pick a random sensing matrix which is shown to have the restricted isometry property with high probability.
 In the present paper, we use a similar strategy with 
a specific choice of random frames and investigate the properties of the corresponding fusion frames, given by unitarily invariant, independently selected subspaces.
We conclude that these are 
with overwhelming probability nearly tight and nearly equidistant, and hence near-optimal
for many applications.

The main results are summarized as follows, see the next section for notation:

If $\{V_j\}_{j=1}^K$ are $s$-dimensional subspaces that are chosen independently at random
in a $N$-dimensional Hilbert space, with a unitarily invariant distribution, and $\epsilon,\delta>0$, $1+\epsilon=(1+\delta)^6$, then
with probability at least 
$$
   1 -2  (1+ \frac{4}{\delta})^N e^{-M\delta^2/4 + M \delta^3/3} - 2 K (1+ \frac{4}{\delta})^s e^{-N\delta^2/4 + N \delta^3/3}
$$
they form an $\epsilon$-nearly $Ks/M$-tight fusion frame.

Under the same assumptions, letting each $P_j$ denote the orthogonal projection onto $V_j$,
and with $1+\epsilon=(1+\delta)^3$,
the subspaces form a nearly equiangular fusion frame, meaning for all $j \ne l$,
$$
  \frac{1}{1+\epsilon} - \epsilon \sqrt \frac{ (1+\epsilon)N}{ s} \le \frac{N}{s^2} \tr[P_j P_l] 
  \le 1+\epsilon(1+ \sqrt{\frac{(1+\epsilon)N}{s}})+\frac{N \epsilon^2}{4 s}  
 $$
  where the failure probability is bounded by the sum of
  $$
    K(K+1)s e^{(1+\delta)s/2-s(s-1)(\delta^2/2-\delta^3/3)/2} \, 
    $$  
  and
  $$
   K(K+1) \bigl( (M+K (1+ \frac{4}{\delta})^s) e^{-N\delta^2/4 + N \delta^3/3} +   (1+ \frac{4}{\delta})^N e^{-M\delta^2/4 + M \delta^3/3} \bigr) \, .
  $$
Examining the terms in these error bounds shows that as $N,s,K \to \infty$ and $s/N \to c>0$ then the failure probability
decays exponentially in $N$ if 
$$
   \frac{3\ln (K+1)}{N} + \frac{s}{N} \ln (1+ \frac 4 \delta ) < \delta^2/4-\delta^3/3
$$
and
$$
   \frac{N}{Ks} \ln (1+ \frac 4 \delta ) < \delta^2/4-\delta^3/3 \, .
$$

This paper is organized as follows: In the next section, we fix notation and recall elementary results.
Section~\ref{sec:tight} demonstrates that random subspaces lead to nearly tight fusion frames.
The final section is dedicated to showing that the subspaces are nearly equiangular.

\section{Frames, Fusion Frames and Riesz sequences}

We briefly review frames, fusion frames and Riesz sequences.

\begin{definition}
A family of vectors $\{\varphi_j\}_{i\in J}$ in a  real or complex Hilbert space $\cH$ is a {\it frame} for
$\cH$ if there are constants $0<A\le B < \infty$ so that for all $x \in \cH$ we have
\[ A\|x\|^2 \le \sum_{j\in J}|\langle x,\varphi_j\rangle|^2 \le B \|x\|^2.\]
If $A=B$ then we say that the frame is $A${\it -tight}, and if $A=B=1$, it is a {\it Parseval frame}.
If there is $c>0$ and $\|\varphi_j\|=c$ for all $j \in J$ then it is called an {\it equal norm frame}, and if $c=1$ it is
a {\it unit norm frame}.   The {\it analysis operator} of the frame $T:\cH
\rightarrow \ell_2(J)$ is  given by
$ (Tx)_j  =  \langle x,\varphi_j\rangle $.
The {\it synthesis operator} $T^*$ is the (Hilbert) adjoint of $T$
and the {\it frame operator} is the positive self-adjoint invertible operator
$S=T^*T$. 
\end{definition}
We recall that if $\{\varphi_j\}_{j \in J}$ is a frame with frame operator $S$ then $\{S^{-1/2}\varphi_j\}_{j\in J}$  is a Parseval frame for $\cH$.

While frames assign scalar coefficients to a vector,
fusion frames map it to its components in subspaces \cite{CK04}.

\begin{definition}Given a  real or complex Hilbert space $\cH$ 
and a family of closed subspaces $\{ W_i\}_{i \in K}$ with
associated positive weights $0<v_i$, $i \in K$, then $\{{W}_i,v_i\}_{i \in K}$ 
is a {\em fusion frame} for $\cH$ if
there exist
constants $0 < A \le B < \infty$ such that
\[
A\|x\|^2 \le \sum_{i \in K} v_i^2 \|P_i x\|^2 \le 
B\|x\|^2 \qquad \mbox{for any }x \in \cH,
\]
where each $P_i$ is the orthogonal projection onto $\mathcal{W}_i$. 
A fusion frame is called {\em tight} if $A$ and $B$ can be chosen
to be equal, and {\em Parseval} if $A=B=1$.
For $\epsilon>0$, the fusion 
frame is $\epsilon${\em -nearly tight}
if there is a constant $C$ so that $A=\frac{1}{1+\epsilon}C,\ B=(1+\epsilon)C$.
The fusion frame is {\em equi-dimensional} if all its subspaces $W_i$ have the same dimension.
If $\{W_i\}_{i\in K}$ are closed subspaces of $\cH$, we define the space
\[ \bigoplus_{i\in K} W_i = \{ \psi=(\psi_i)_{i\in K} |\ \psi_i \in W_i, \langle \psi, \psi \rangle < \infty \},\]
with the inner product  given by
\[ \left \langle (\psi_i)_{i\in K},(\phi_i)_{i\in K}\right \rangle = \sum_{i\in K}\langle
\psi_i,\phi_i\rangle.\]
The {\it analysis operator} of the fusion frame is the operator 
\[T:\cH \rightarrow  \bigoplus_{i\in K} W_i,\]
given by
\[ Tx = (v_i P_i x)_{i\in K}.\]
The {\it synthesis operator} of the fusion frame is $T^*$ and is given by
\[ T^*(\psi_i)_{i\in K} = \sum_{i\in K} v_i\psi_i.\]
In analogy with frames, the fusion frame operator is the positive, self-adjoint and invertible operator
$S=T^* T$.
\end{definition}

We also need the notion of $\epsilon$-Riesz sequences.
We choose the convention from \cite{BCC12} for
tight Riesz sequences which is convenient for orthonormalization.

\begin{definition} \label{D5} 
A family of vectors $\{\varphi_i\}_{i=1}^N$ in a Hilbert space $\cH$ is a Riesz basic sequence with 
lower (resp. upper) Riesz bounds $0<A\le B< \infty$ if for all scalars $\{a_i\}_{i=1}^N$
we have
\[ A\sum_{i=1}^N|a_i|^2 \le \|\sum_{i=1}^Na_i\varphi_i\|^2 \le B \sum_{i=1}^N|a_i|^2.\]
This family of vectors is $\epsilon$-Riesz basic 
if for all scalars $\{a_i\}_{i=1}^N$ we have
\[ \frac{1}{1+\epsilon}\sum_{i=1}^N|a_i|^2 \le \|\sum_{i=1}^Na_i\varphi_i\|^2 \le
(1+\epsilon)\sum_{i=1}^N|a_i|^2.\]
\end{definition}

\section{Random fusion frames are nearly tight} \label{sec:tight}

We first recall that random Gaussian vectors form a frame that can be partitioned into nearly orthonormal systems. 

\begin{lemma}
Let $X$ be a random $N\times M$ matrix whose entries are independent, standard-normal distributed random variables,
and let $X_j$ denote the vector containing the entries of the $j$th row.
Let $u \in \mathbb R^M$, $\|u\|=1$, and $Z= \frac 1 N \sum_{j=1}^N |\langle u, X_j \rangle |^2$ then 
$$
 \mathbb P ( Z \ge  1 + \delta ) \le e^{-N\delta^2/4 + N \delta^3/6}
 $$
and
$$
  \mathbb P ( Z \le  \frac{1}{1 + \delta} ) \le e^{-N\delta^2/4 + N \delta^3/3}
$$
\end{lemma}
\begin{proof} The distribution of the row vectors is unitarily invariant, so
without loss of generality, we can set $u=e_1$, the first vector of the canonical basis for $\mathbb R^M$. The sum 
$Z= \frac 1 N \sum_{j=1}^N X_{j,1}^2$ 
is up to the normalization factor $1/N$ chi-squared distributed with  $N$ degrees of freedom.
The usual combination of the Laplace transform and the Chernoff bound gives measure concentration. 
We have 
$$
  \mathbb P(Z \ge 1+\delta ) \le (1-t)^{-N/2} e^{- t N (1+\delta)/2} 
$$
for any $0 \le t <1$, so after choosing $t = \delta/(1+\delta)$ and truncating the Taylor expansion of the exponential
$$
  \mathbb P(Z \ge 1+\delta )  \le e^{N \ln(1+\delta)/2 -N \delta/2} \le e^{-N\delta^2/4 + N \delta^3/6} \, .
$$
Similarly, for $t \ge 0$,
$$
   \mathbb P(Z \le \frac{1}{1+\delta} ) \le e^{- N \ln(1+t)/2 + t N (1+\delta)^{-1}/2}
$$
so setting $t = \delta$ and comparing the terms of the Taylor series in the exponent
gives the desired bound.
\end{proof}

Choosing $u$ among the $M$ canonical basis vectors in $\mathbb R^M$ together with a union bound shows that the norms of
all the columns are nearly a constant.

\begin{corollary}  \label{cor:norm}
With the random $N\times M$ matrix $X$ as above, 
$$
   \mathbb P( \frac{1}{1+\delta} \le \frac 1 N \sum_{j=1}^N X_{j,l}^2 \le 1 + \delta \mbox{ for all } l ) 
   \ge 1 - 2 M e^{-N\delta^2/4 + N \delta^3/3} \, .
$$
\end{corollary}

We derive a stronger consequence with an argument similar to the exposition in Baraniuk et al. \cite{BDDVW},
as presented in \cite{BB-FOC}: With high probability, sufficiently small subsets of the column vectors of 
$X$ have nearly tight upper and lower Riesz bounds.

\begin{lemma} \label{lem:riesz}
Let $W=\mathrm{span}\{e_{j_1}, e_{j_2}, \dots e_{j_s}\}$
in $\mathbb R^M$, let $\{X_j\}_{j=1}^N$ be a random family of vectors in $\mathbb R^M$ as above
and let $0<\delta<2$, then the set 
\begin{equation*} \label{ineq:Pconci}
    \mathcal X_\le = \{X: Z(y) \le \frac{1}{(1+\delta)^{3}} \|y\|^2 \mbox{ for all } y \in W \}
  \end{equation*}
  defined by the random variables
$Z(y)= \frac 1 N \sum_{j=1}^N |\langle y, X_j \rangle |^2$ for $y \in W$
has probability
$$
  \mathbb P ( \mathcal X_\le) \le  (1+ \frac{4}{\delta})^s e^{-N\delta^2/4 + N \delta^3/3} \, .
$$
Moreover, if $\delta < 1$, then
$$
  \mathcal X_\ge = \{X: Z(y) \ge {(1+\delta)^{3}} \|y\|^2 \mbox{ for all } y \in W \}
$$
has the same upper bound for its probability as $\mathcal X_\le$.
\end{lemma}

\begin{proof}
 Using the triangle inequality
 and Lipschitz continuity of the norm we can bootstrap from
 a net $S$ with $\displaystyle\min_{w\in S}\|y-w\|\le \frac{\delta}{2}$ for all $y\in W, \|y\|=1$. By a volume inequality for sphere packings, we know there is such an $S$ with
 cardinality
\[
 |S|\le\left(1+\frac{4}{\delta}\right)^s \, .
\]
Applying a union bound  
for the probability of the complement of
\[
 \mathcal X_S = \{X: Z(w)\ge\frac{1}{1+\delta} \|w\|^2 \,  \mbox{for all} \,  w\in S \}
\]
we get
$$
  \mathbb P(\mathcal X_S) \ge 1- (1+\frac{4}{\delta})^s e^{-N\delta^2/4 + N \delta^3/3}
$$
Now let $a$ be the smallest number such that 
$Z^{1/2}(y)\ge\frac{1}{(1+a)^{1/2}}\|y\|$ holds for all $y\in W$.
We show $1+a\le(1+\delta)^3$. To see this, let $y\in W, \|y\|=1$ and pick $w\in S,\|x-w\|\le \frac{\delta} 2$.
Then, using the triangle inequality yields
\[
 Z^{1/2}(y)\ge
  | Z^{1/2}(w)-Z^{1/2}(y-w) |\ge \frac{1}{(1+\delta)^{1/2}}-\frac{\delta/2}{(1+a)^{1/2}}  \, .
\]
Since the right hand side of the inequality chain is independent of $y$, according to the definition of $a$ we obtain
\begin{eqnarray*}
 & &\frac{1}{(1+a)^{1/2}}\ge\frac{1}{(1+\delta)^{1/2}}-\frac{\delta/2}{(1+a)^{1/2}} \, .
 \end{eqnarray*}
Solving for  $(1+a)^{-1}$ and further estimating gives
\begin{eqnarray*}
 \frac{1}{1+a}\ge\frac{1}{(1+{\delta})^{3}} \, .
 \end{eqnarray*}

For the second inequality, choose
\[
 \mathcal X'_S = \{X: Z(w)\le ({1+\delta}) \|w\|^2 \,  \mbox{for all} \,  w\in S \}
\]
and establish the same bound for its probability as for $\mathcal X_S$.
Let $a$ be smallest such that
$$
    Z^{1/2}(y)\le (1+a)^{1/2}\|y\| \, , y \in W.
$$
We again use the triangle inequality 
to obtain that if $X \in \mathcal X'_{S}$ then for any $y\in W$ with $\|y\|=1$
$$
   Z^{1/2}(y) \le (1+\delta)^{1/2} +  (1+a)^{1/2} \frac \delta 2 \, .
$$
Again by definition
$$
  (1+ a)^{1/2} \le (1 + \delta)^{1/2} +(1+ a)^{1/2} \frac \delta 2 
$$
so
$$
  1+a \le \frac{1 + \delta}{(1-\delta/2)^2}
$$
and if $\delta< 1$ then $(1-\delta/2)^{-1} \le 1 + \delta $, which implies
$$
  1+a \le (1+ \delta)^3 \, .    
  $$
\end{proof}

A union bound gives a lower bound for the probability that all
subsets in a partition of the column vectors of $X$ have good Riesz bounds.

\begin{corollary} \label{thm:riesz_conc}
Let $X$ be a random $N\times M$ matrix whose entries are independent, standard-normal distributed random variables, 
$0 < \delta < 1$, and
let $\{1 , 2, \dots, M\}$ be partitioned into
sets $\{J_k\}_{k=1}^K$ of maximal size $\max_k |J_k| \le s$, then with probability
$$
 1 - 2 K (1+ \frac{4}{\delta})^s e^{-N\delta^2/4 + N \delta^3/3} \, .
$$
for each $k$ the set of rescaled column vectors $\{(X_{j,l}/\sqrt N)_{j=1}^N \}_{l \in J_k}$  forms  an
$\epsilon$-Riesz sequence with $\epsilon = (1+\delta)^3-1$.
\end{corollary}
\begin{proof}
If the matrix is multiplied by the normalization factor, then a fixed set of columns $\{(X_{j,l}/\sqrt N)_{j=1}^N \}_{l \in J_k}$ indexed by $J_k$ in the partition
with $|J_k| \le s$ is by the preceding lemma $\epsilon$-Riesz with $1+\epsilon=(1+\delta)^3$ with probability of at least
$1-2(1+ \frac{4}{\delta})^s e^{-N\delta^2/4 + N \delta^3/3}$. Applying the union bound for all $K$ sets in the partition
gives the claimed estimate.
\end{proof}

The same proof allows us to establish frame bounds if we think of $X$ as the analysis operator.
If we choose a trivial partition with $s=M\le N$ then the above lemma states that the family of 
row vectors $\{X_j\}_{j=1}^N$ with standard normal entries forms 
with high probability a nearly tight frame
for $\mathbb R^M$. We swap $N$ and $M$ for later notational convenience.

\begin{corollary} \label{cor:frame}
Let $X$ be a random $M\times N$ matrix whose entries are independent, standard-normal distributed random variables,
then with probability at least
$$
    1-2  (1+ \frac{4}{\delta})^N e^{-M\delta^2/4 + M \delta^3/3} \, .
$$
 the row vectors $\{X_j\}_{j=1}^M$ of $X$  form
 a frame for $\mathbb R^N$ with lower and upper frame bounds
$$
   \frac{M}{(1+\delta)^3} \|x\|^2 \le \sum_{j=1}^M |\langle x, X_j \rangle |^2 \le M (1+\delta)^3 \|x\|^2 , \quad x \in \mathbb R^N \, .
$$
\end{corollary}

\begin{thm} \label{thm:riesz}
Given a family of random, independent subspaces $\{V_k\}_{k=1}^K$ of $\mathbb R^N$, whose distribution is invariant under unitaries
and whose dimensions $s_k = \dim V_k$ are bounded by $s_k \le s$ for all $k$, and let $M = \sum_{k=1}^K s_k$ then with probability at least 
$$
   1 -2  (1+ \frac{4}{\delta})^N e^{-M\delta^2/4 + M \delta^3/3} - 2 K (1+ \frac{4}{\delta})^s e^{-N\delta^2/4 + N \delta^3/3}
$$
the fusion frame $\{V_j, 1\}_{j=1}^K$ has upper and lower frame bounds $ M/(N(1+\delta)^6)$  and $M (1+\delta)^6/N$.
\end{thm} 
\begin{proof}
The preceding lemma implies that the random Gaussian $M\times N$ matrix $\frac 1 {\sqrt{N}} X$ is with high probability the analysis operator of a frame that is 
nearly $M/N$-tight, because
$$
   M/(N(1+\delta)^3) \|x\|^2 \le \frac{1}{N} \sum_{j=1}^M |\langle x, X_j \rangle |^2 \le M (1+\delta)^3/N 
$$
 for all $ x \in \mathbb R^N$ 
except for a set of probability $2(1+ \frac{4}{\delta})^N e^{-M\delta^2/4 + M \delta^3/3}$.
 
 Moreover, let $\{1, 2, \dots, M\}$ be partitioned into subsets $\{J_k\}_{k=1}^K$ of size $\max_i |J_i| \le s$,
then a union bound shows that
with probability bounded below by
$$
   1 -2  K (1+ \frac{4}{\delta})^s e^{-N\delta^2/4 + N \delta^3/3}
$$
each set of rescaled row vectors $\{X_j/\sqrt N \}_{j\in J_k}$  has upper and lower Riesz bounds
$(1+\delta)^3$ and $(1+\delta)^{-3}$, respectively.

Orthonormalizing the Riesz sequences then changes the frame bounds by at most a factor of $(1+\delta)^{\pm 3}$.
The resulting frame is by construction partitioned into orthonormal systems, and thus equivalent to a fusion frame
with the same frame bounds. 
Combining the probabilities for the failure of the frame bounds and of the Riesz bounds gives the stated bound.
\end{proof}

\section{Random fusion frames are nearly equiangular} \label{sec:equiangular}

Using a variation  of the strategy by Dasgupta and Gupta \cite{DG03}, we obtain measure concentration for
an average of projected norms, which implies that with a proper scaling of dimensions, random subspaces become nearly equiangular.

\begin{lemma}
Given $\{X_{j,l}: 1 \le j \le s, 1 \le l \le N\}$, independent identically standard-normal distributed random variables, $N\ge s$,
and $0< \beta < 1$ then
$$
     \mathbb P_\le \equiv \mathbb P(N\sum_{j=1}^s \sum_{l=1}^s X_{j,l}^2 \le \beta s^2 \sum_{l=1}^N X_{1,l}^2 )
     \le e^{s(s-1)(\ln \beta) /2 +s/2+(1-\beta)s^2/2\beta} \, .
$$
Similarly, if $\beta > 1$, then
$$
     \mathbb P _\ge \equiv \mathbb P(N\sum_{j=1}^s \sum_{l=1}^s X_{j,l}^2 \ge \beta s^2 \sum_{l=1}^N X_{1,l}^2 )
     \le e^{s(s-1)(\ln \beta) /2 +s/2+(1-\beta)s^2/2} \, .
$$
\end{lemma}
\begin{proof}
The first probability under consideration is equal to
\begin{align*}
   &\mathbb P ( e^{\beta s^2 \sum_{l=1}^N X_{1,l}^2 - N \sum_{j,l=1}^s X_{j,l}^2} \ge 1 )\\
  & \le \mathbb E [e^{t(\beta s^2 \sum_{l=1}^N X_{1,l}^2 - N \sum_{j,l=1}^s X_{j,l}^2)}]\\
   & = \mathbb E[ e^{t \beta s^2 X_{1,1}^2}]^{N-s}  \mathbb E[ e^{t (\beta s^2 -N) X_{1,1}^2}]^{s} 
    \mathbb E[ e^{-t N X_{1,1}^2}]^{s(s-1)}  \\
    &\le \mathbb E[ e^{t \beta s^2 X_{1,1}^2} ]^{N}  \mathbb E[ e^{-t N X_{1,1}^2}]^{s(s-1)}  
    = (1-2 t \beta s^2)^{-N/2} (1+2 t N)^{-s(s-1)/2}
\end{align*}
The right hand side is minimized with
$$
   t = \frac{(1-\beta)s-1}{2 \beta s(N+s(s-1))}
$$
and inserting this in the expression
\begin{align*}
    \mathbb P_\le & \le \left( \frac{N + \beta s^2}{N+(s-1)s} \right)^{-N/2} \left( \frac{(s-1)(N + \beta s^2)}{\beta s (N+(s-1)s)} \right)^{-s(s-1)/2} \\
     & = (N+(s-1)s)^{s(s-1)/2+N/2} (N+\beta s^2)^{-s(s-1)/2-N/2}\left( \frac{s-1}{\beta s} \right)^{-s(s-1)/2} \\
     & \le \left(1 + \frac{(1-\beta)s^2}{N+\beta s^2} \right)^{s(s-1)/2+N/2} \left(1-\frac 1 s \right)^{-s(s-1)/2}  \beta^{s(s-1)/2} \\
      & \le \beta^{s(s-1)/2} e^{s/2} e^{(1-\beta)s^2(s(s-1)/2+N/2)/( N+\beta s^2)} \le \beta^{s(s-1)/2} e^{s/2} e^{(1-\beta)s^2/2\beta}\\
      & = e^{s(s-1)(\ln \beta) /2 +s/2+(1-\beta)s^2/2\beta}
\end{align*}
If $\beta >1$, then we get
\begin{align*}
   &\mathbb P ( e^{\beta s^2 \sum_{l=1}^N X_{1,l}^2 - N \sum_{j,l=1}^s X_{j,l}^2} \le 1 )\\
    &\le 
     (1+2 t \beta s^2)^{-N/2} (1-2 t N)^{-s(s-1)/2}
\end{align*}
Now the optimal choice for $t$ is
$$
  t = - \frac{(1-\beta)s-1}{2 \beta s(N+s(s-1))}
$$ 
and the result follows 
\begin{align*}
    \mathbb P_\ge & \le \left( \frac{N + \beta s^2}{N+(s-1)s} \right)^{-N/2} \left( \frac{(s-1)(N + \beta s^2)}{\beta s (N+(s-1)s)} \right)^{-s(s-1)/2} \\
      & \le \beta^{s(s-1)/2} e^{s/2} e^{(1-\beta)s^2(s(s-1)/2+N/2)/( N+\beta s^2)} \le \beta^{s(s-1)/2} e^{s/2} e^{(1-\beta)s^2/2}\\
      & = e^{s(s-1)(\ln \beta) /2 +s/2+(1-\beta)s^2/2} \, .
\end{align*}
\end{proof}

Setting $t=1+\delta$ or $t=(1+\delta)^{-1}$ gives the following estimate.

\begin{corollary}
Given $\{X_{j,l}: 1 \le j \le s, 1 \le l \le N\}$, independent identically standard-normal distributed random variables,  $N \ge s$
and $\epsilon>0$, then
\begin{align*}
    \mathbb P(N\sum_{j=1}^s \sum_{l=1}^s X_{j,l}^2 \le \frac{1}{1+\delta} s^2 \sum_{l=1}^N X_{1,l}^2 )
     & \le e^{-s(s-1)(\ln (1+\delta)) /2 +s/2+\delta s^2/2}\\
    &  \le e^{(1+\delta)s/2-s(s-1)(\delta^2/2-\delta^3/3)/2}\, 
\end{align*}
and
\begin{align*}
    \mathbb P(N\sum_{j=1}^s \sum_{l=1}^s X_{j,l}^2 \ge ({1+\delta}) s^2 \sum_{l=1}^N X_{1,l}^2 )
     & \le e^{s(s-1)(\ln (1+\delta)) /2 +s/2-\delta s^2/2}\\
    &  \le e^{(1+\delta)s/2-s(s-1)(\delta^2/2-\delta^3/3)/2}\, .
\end{align*}
\end{corollary}

For the next result  it is sometimes more convenient to work with random matrices whose column vectors are normalized. We 
will alternate between the random $N\times M$ matrix $X$ with standard normal entries and the random matrix $x$  which is obtained by
$$
    x_{j,l} = X_{j,l} /( \sum_{i=1}^N X_{i,l}^2 )^{1/2} \, 
$$
and consequently has column vectors that are independent, uniformly distributed on the unit sphere in $\mathbb R^N$.

\begin{thm} \label{thm:conc_proj}
Let $\{x_1, x_2, \dots, x_s\}$ be vectors in $\mathbb R^N$, drawn independently according to a uniform distribution on the unit sphere.
If $V$ is a fixed subspace of dimension $s<N$
and $P_V$ is the orthogonal projection onto $V$, then
$$
   \mathbb P (\frac{1}{1+\delta} \le \sum_{i=1}^s \frac{N}{s^2}\|P_V x_i\|^2 
   \le 1+\delta )
   \ge 1-2s e^{(1+\delta)s/2-s(s-1)(\delta^2/2-\delta^3/3)/2} .
$$
\end{thm}
\begin{proof}
This follows from the fact that mapping a standard normal Gaussian random vector 
to the unit vector in its span leads to a uniform distribution on the sphere. The preceding probability
estimates are unchanged by scaling the vectors on both sides. 
Let $Z_j = \sum_{l=1}^N X_{j,l}^2$ then the theorem we wish to prove is equivalent to
 $$
  \mathbb P(\frac{1}{1+\delta}   \le \frac{N}{s^2} \sum_{j=1}^s \frac{1}{Z_j} \sum_{l=1}^s X_{j,l}^2 \le (1+\delta)  )
  \ge 1  -2s e^{(1+\delta)s/2-s(s-1)(\delta^2/2-\delta^3/3)/2} .
$$
To deduce this, we 
use a union bound again, which implies with the preceding estimates that
\begin{align*}
  &\mathbb P( \frac{1}{1+\delta} \max_j  Z_j \le \frac{N}{s^2} \sum_{j=1}^s \sum_{l=1}^s X_{j,l}^2 \le (1+\delta)  \min_j  Z_j )\\
  &\ge 1  -2s e^{(1+\delta)s/2-s(s-1)(\delta^2/2-\delta^3/3)/2} .
\end{align*}
Now $\min_i \sum_l X_{i,l}^2 \le Z_j \le \max_i \sum_l X_{i,l}^2$ establishes the bound.
\end{proof}

Next, we prepare the result on the equiangularity of random fusion frames. 

\begin{lemma}
Let $M=Ks$, $M \ge N\ge s$, $0<\delta<1$, $\epsilon=(1+\delta)^3-1$ and $\{x_i\}_{i=1}^M$ be independent random vectors  in $\mathbb R^N$
that are uniformly distributed on the unit sphere. Let $\{P_k\}_{k=1}^K$
be the orthogonal projection onto the span of $\{x_i \}_{i \in J_k}$, where $\{J_k\}$ partitions 
the index set and each $J_k$ has size $s$, then for fixed $k\ne l$,
\begin{align*}
  & \mathbb P ( \frac{1}{1+\epsilon} - \epsilon \sqrt \frac{ (1+\epsilon)N}{ s} \le \frac{N}{s^2} \tr[P_l P_k] 
  \le 1+\epsilon(1+ \sqrt{\frac{(1+\epsilon)N}{s}})+\frac{N \epsilon^2}{4 s}  )\\
  & \ge 1 - R_1 -R_2 
  \end{align*}
  where the failure probability is bounded by the sum of
  $$
     R_1 = 2s e^{(1+\delta)s/2-s(s-1)(\delta^2/2-\delta^3/3)/2} \, .
    $$ 
  
  and
  $$
   R_2 = 2 (M+K (1+ \frac{4}{\delta})^s) e^{-N\delta^2/4 + N \delta^3/3} + 2  (1+ \frac{4}{\delta})^N e^{-M\delta^2/4 + M \delta^3/3} \, .
  $$
\end{lemma}
\begin{proof}
A random, uniformly distributed 
$s$-dimensional subspace
is realized by taking the span of independent, identically uniformly distributed random unit vectors $\{x_i\}_{i=1}^s$.
Correspondingly, the orthogonal projection is obtained via orthonormalizing these random vectors with the square root of the pseudoinverse
$S^\dagger$ of $S=\sum_{j=1}^s x_j \otimes x_j^*$.
As a consequence, the trace is identical to
$$
\tr[P_l P_k] = \sum_{j=1}^s \tr[P_l (S^\dagger)^{1/2}  x_j \otimes x_j^* (S^\dagger)^{1/2} ] = \sum_{j=1}^s \|P_l (S^\dagger)^{1/2} x_j \|^2 \, .
$$
With the triangle inequality, we split this expression into three parts that we estimate separately,
\begin{align*}
 Q_1 - Q_2 + Q_3 \le 
         \sum_{i \in J_k} \| P_l (S^\dagger)^{1/2} x_i \|^2 \le
         Q_1 + Q_2 + Q_3
 \end{align*}
 with
 $$
    Q_1 = \sum_{i \in J_k} \| P_l x_i \|^2 , \quad Q_2 = 2 \sum_{i \in J_k}  \| P_l x_i \| \|P_l  ((S^\dagger)^{1/2} x_i - x_i) \| 
 $$
 and
 $$
    \, \quad Q_3 = \sum_i \|P_l  ((S^\dagger)^{1/2} x_i - x_i) \|^2 \, .
 $$
 The Cauchy-Schwarz inequality gives $Q_2 \le 2 ( Q_1 Q_3)^{1/2}$ so it is enough to control $Q_1$ and $Q_3$.
 
 The quantity $Q_1$ is concentrated near $s^2/N$ by Theorem~\ref{thm:conc_proj}, which gives a lower probability bound of $1-R_1$ for
 the set with $\frac{1}{1+\epsilon} \le \frac{1}{1+\delta} \le NQ_1/s^2 \le 1+\delta \le 1+\epsilon$.
 The third quantity
 is with probability $1-R_2$ small by Theorem~\ref{thm:riesz_conc}, because if $\epsilon = (1+\delta)^3-1$
 then $\frac{1}{1+\epsilon} \le S \le 1+\epsilon$ and if $\|x\|=1$, then
 $$
      \| (S^\dagger)^{1/2} x - x \|^2 \le (\frac{1}{\sqrt{1+\epsilon}} - 1)^2  \le (\sqrt{1+\epsilon} - 1)^2 \le \epsilon^2/4 \, .
 $$
 Thus, apart from a set of probability given in Theorem~\ref{thm:riesz_conc}, $Q_3 \le s \epsilon^2/4$ and $2(Q_1 Q_3)^{1/2} \le (Q_1 s)^{1/2} \epsilon$.
  
  Next, if $\frac{1}{1+\epsilon} \frac{s^2}{N} \le Q_1 \le (1+\epsilon) \frac{s^2}{N}$ then
  $$
     \frac{1}{1+\epsilon} \frac{s^2}{N} - (1+\epsilon)^{1/2} \frac{s^{3/2}\epsilon}{\sqrt N} \le Q_1 - Q_2 \le \sum_{i \in J_k} \| P_l (S^\dagger)^{1/2} x_i \|^2
  $$
  and
  $$
    \sum_{i \in J_k} \| P_l (S^\dagger)^{1/2} x_i \|^2 \le Q_1 + Q_2 + Q_3 \le (1+\epsilon) \frac{s^2}{N} + (1+\epsilon)^{1/2} \frac{s^{3/2}\epsilon }{\sqrt N} + s \epsilon^2/4 \, .
  $$
 \end{proof}

By a union bound over all pairings of subspaces, we get the following, more qualitative estimate:

\begin{thm} Let $K,s,N \in \mathbb N$, $M=Ks \ge N \ge s$.
For any $c>1$ there exists $\epsilon_0>0$ such that for all $0<\epsilon<\epsilon_0$, if 
 $\{V_k\}_{k=1}^K$ is a family of $s$-dimensional subspaces selected independently at random according to the  unitarily invariant distribution 
 in $\mathbb R^N$
and $\{P_k\}_{k=1}^K$ the corresponding orthogonal projections, then
the set
$$
 \mathcal X = \{1 - c \epsilon(1+ \sqrt{\frac N s}) \le \frac{N}{s^2}  \tr [ P_k P_l ]\le 1+ c \epsilon(1+\sqrt{\frac N s}) +\frac{N\epsilon^2}{4s} \mbox{ for all } k\ne l  \}
$$
has probability
\begin{align*}
   \mathbb P (\mathcal X)
 \ge
  1- (R_1+R_2)K(K-1)/2
\end{align*}
with $R_1$ and $R_2$ as in the preceding lemma.
\end{thm}
\begin{proof}
If $\epsilon$ is sufficiently small, then the
remainder for the series expansion,  $1/(1+\epsilon)-1+\epsilon+(\sqrt{1+\epsilon}-1-\epsilon/2)\sqrt{N/s}$ 
is bounded by $(c-1)\epsilon+c\epsilon\sqrt{N/s}$. A union bound for the $K(K-1)/2$ pairings gives
the probability bound.
\end{proof}

%

\begin{corollary}
Let $K,N,s$ and $\delta$ be as above. 
If $N,s\to \infty$, $s/N\to c>0$, and $K$ is such that
$$
   \frac{3\ln (K+1)}{N} + \frac{s}{N} \ln (1+ \frac 4 \delta ) < \delta^2/4-\delta^3/3
$$
and
$$
   \frac{N}{Ks} \ln (1+ \frac 4 \delta ) < \delta^2/4-\delta^3/3
$$
then the upper bound $R_1+R_2$ for the probability of $\mathcal X$ 
in the preceding theorem decays exponentially in $N$.
\end{corollary}
\begin{proof}
As $s/N$ remains bounded away from zero as $N \to \infty$, $R_1$ decays exponentially.
Examining the exponents of the terms in $R_2$ shows that if the conditions on $K$, $s$ and $N$
are satisfied then it decays exponentially as $N \to \infty$ as well.
\end{proof}

\end{document}